\documentclass[12pt]{article}
\usepackage{graphicx} 
\usepackage[a4paper, total={15cm, 22.5cm}]{geometry} 
\usepackage{amsmath}
\usepackage{amssymb}
\usepackage{amsthm}
\usepackage{dsfont}
\usepackage{svg}
\usepackage{caption}
\usepackage{subcaption}
\usepackage{wrapfig}    
\usepackage{enumerate}

\newtheorem{theorem}{Theorem}[section]
\newtheorem{lemma}[theorem]{Lemma}
\newtheorem{proposition}[theorem]{Proposition}

\newtheorem{definition}[theorem]{Definition}

\title{New Constructions of Cubature Formulas on Wiener Space}   

\author{Timothy Herschell}             
\date {}

\begin{document}

\setcounter{secnumdepth}{3}
\setcounter{tocdepth}{3}

\maketitle
\begin{abstract}
    Building on techniques developed by Lyons and Victoir \cite{LyonsTerry2004CoWs}, we present the first explicit construction of a degree-7 cubature formula for Wiener space over R3. We then examine and compare two approaches for computing cubature approximations: one based on the stochastic Taylor expansion and the other on the Log-ODE method. Our numerical experiments illustrate how the cubature degree influences the order of convergence and demonstrate the utility of cubature methods for weak approximations of stochastic differential equations (SDEs).

    These results were originally part of a Master’s thesis and are provided here as context and a reference point for subsequent work. A more general construction in arbitrary dimensions has since been obtained in \cite{highdegreecubaturewienerspace} using different techniques.
\end{abstract}

\section{Introduction}
The conception of stochastic analysis brought forth a multitude of techniques that revolutionised our understanding of random processes. It emerged that, unlike classical analysis, there is no unified theory of integration in stochastic analysis. Multiple ideologies, such as the popular It\^o and Stratonovich formulations, could be equally justifiable yet give seemingly different results. Furthermore, when processes do not meet strict smoothness conditions, the stochastic differential equations they govern may have discontinuous solution maps. Such paths are called \textit{rough paths} and were first tackled by Lyons in his landmark paper \cite{alma991025219949807026}, which has served as the inspiration for countless breakthroughs in the field since.

This work will concern itself with one particularly pivotal example of a rough path, Brownian motion. Brownian motion has a H\"older exponent of $\alpha=\frac{1}{2}$, so is on the cusp of the rough path classification. We will use Brownian motion in $\mathbb{R}^d$, denoted $B_t$, as a driving signal in controlled differential equations defined over some domain $t\in[0,T]$. That is,
\begin{equation}\label{eq:CDE-intro}
    \text{d}X_t=f(X_t,t)\circ\text{d}B_t,
\end{equation}
where $f$ is some bounded, Lipschitz function. It becomes useful to augment the Brownian motion with an extra dimension representing time, to incorporate a constant drift factor into (\ref{eq:CDE-intro}).

When $d>1$, equations of the form (\ref{eq:CDE-intro}) often admit no analytical solutions. This led to the conception of ``Cubature on Wiener Space", written by Lyons and Victoir \cite{LyonsTerry2004CoWs}, which has inspired this work. Lyons and Victoir developed a numerical approximation scheme (referred to as \textit{cubature}) for solving equations of the form (\ref{eq:CDE-intro}) in expectation. For brevity, we will sometimes refer to a solution for $\mathbb{E}[X_t]$ as a \textit{weak solution}. Heuristically, a cubature formula on Wiener space is a series of (augmented) Brownian motion sample paths that approximate the expectation of Brownian motion up to some degree $m$. The details of this will require the notion of a \textit{path signature}, a path enhanced with additional feature information that ensures continuity of solution paths. The full path signature is composed of an infinite number of levels, though often truncated after some point, and the cubature formula will be said to have degree $m$ if it matches the signature up to level $m$ in expectation. 

Lyons and Victoir laid out a framework from which to develop explicit constructions of cubature formulae on Wiener space and included examples of degree 3 and degree 5 formulae. Despite this, no explicit construction has been published for the next degree of interest, degree 7, at least not in general dimension $d$. There are various publications for certain cases including the cases $d=1$ and $d=2$ given by Litterer \cite{LittererThesis} and an implicit construction for general dimension given by Gyurk\'o and Lyons \cite{alma991025225271007026}. This work will present a degree 7 construction explicitly.

The second half of this work will explore the importance of higher degree cubature formulae. We will take two stochastic differential equation solving techniques, namely the stochastic Taylor expansion \cite{alma991022208701907026} and the Log-ODE method \cite{CastellFabienne1995Aeam}, and apply them to produce weak solutions of (\ref{eq:CDE-intro}). Both methods exhibit error terms bounded above by $CT^\frac{m+1}{2}$, for some constant $C$ that is independent of time $T$. Therefore, approximating over smaller and smaller sub-intervals of $[0,T]$ should see the error converge to zero at an exponential rate proportional to degree $m$.

The general structure of this work will be as follows. Chapter \ref{chap:prelim} will present some preliminary knowledge and notation to be used throughout the remainder of the work. In particular, we will introduce cubature formulae in the context of a positive measure and lay out robust definitions of Brownian motion and Wiener space.

Chapter \ref{chap:cub-def} will be dedicated to giving a formal definition of a cubature formula on Wiener space. This will require a fundamental knowledge of tensor algebras and free lie algebras, detailed in Section \ref{sec:algebras}, and of path signatures, detailed in Section \ref{sec:path-sigs}. The definition of Wiener space cubature will be given in Section \ref{sec:cub-def} and will appear in two forms, firstly in terms of sample paths and secondly as elements of the free lie algebra, known as \textit{lie polynomials}.

Chapter \ref{chap:construct} will contain explicit cubature constructions for degree 3, degree 5 and degree 7. The former two are taken from Lyons and Victoir \cite{LyonsTerry2004CoWs}, and the latter is a newly developed construction. These constructions are \textit{efficient} in the sense that they grow in size at a rate proportional to a polynomial in the Wiener space dimension $d$.

Chapter \ref{chap:uses} will detail two cubature methods: the Taylor method and the Log-ODE method. We will include formal results regarding the error term associated with each method. The chapter will conclude with a short discussion about approximating the expected signature of a process satisfying (\ref{eq:CDE-intro}).
\section{Preliminaries}\label{chap:prelim}
This chapter is dedicated to laying out some preliminary knowledge that will be applied in the remainder of the work. The main function of the chapter is to clearly lay down notational standards that will be maintained throughout. 

\subsection{Cubature formula for a positive measure}
A vast number of differential equations have no analytic solutions and for this reason, almost any applied mathematician will rely upon suitable numerical approximations at some stage. Some authors may refer to numerical integration by means of a \textit{quadrature formula} (or if the integrand has dimension two or more, a \textit{cubature formula}). In the context of a positive measure, the precise definition of a cubature formula is the following.
\begin{definition}[Cubature Formula for a Positive Measure]
    Let $\mu$ be a positive measure on $\mathbb{R}^d$, and let $m\in\mathbb{Z}$. A set of points, $x_1,\cdots,x_n\in\mathbb{R}^d$ and a set of weights $\lambda_1,\cdots,\lambda_n\in\mathbb{R}$ are said to define a cubature formula of degree $m$ if,
    \[
    \int_{\mathbb{R}^d}P(x)\mu(\text{d}x)=\sum_{i=1}^n\lambda_iP(x_i)
    \]
    for every polynomial $P\in\mathbb{R}_m[x_1,\cdots,x_d]$ with degree no more than $m$. If $d=1$, then we instead use the term \textit{quadrature formula}.
\end{definition}
Positive measures will not be the focus of this work, but they provide a backstory to the concept of cubature formulae on Wiener space. Let us first explore one introductory example, the Bernoulli measure. This specific example will be used to make explicit constructions in Chapter \ref{chap:construct}.
\begin{proposition}[Cubature formula for the Bernoulli measure]
    Let $m\in\mathbb{N}$. Let $\mu$ be the Bernoulli measure on $\mathbb{R}^d$. That is,
    \[
        \mu (x) = \begin{cases}
            2^{-d}&\text{if }x\in S;\\
            0&\text{otherwise;}
        \end{cases}
    \]
    where $S=\big \{x=(x_1,\cdots,x_d):x_i\in\{-1,1\}\text{ for all }i=1,\cdots,d\big\}$. Then the points $x_i\in\big\{(a_1,\cdots,a_d)\ :\ a_k\in\{-1,1\}\big \}$ and associated weights $\lambda_i=\frac{1}{2^d}$ define a cubature formula of degree $m$ for $\mu$. 
\end{proposition}
\begin{proof}
    Suppose $P\in\mathbb{R}_m[x_1,\cdots,x_d]$ is some monomial $P(x) = x_1^{k_1}\cdots x_d^{k_d}$. Let $\Lambda =(\Lambda_1,\cdots\Lambda_d)$ have a $d$-dimensional Bernoulli distribution. Then,
    \begin{align*}
        \int_{\mathbb{R}^d}P(x)\mu(\text{d}x)&=\mathbb{E}[\Lambda_1^{k_1}\cdots\Lambda_d^{k_d}]\\
        &=\mathds{1}_{k_1\text{ is even}}\cdots\mathds{1}_{k_d\text{ is even}}\\
        &=\sum_{i=1}^{2^d}\lambda_iP(x_i),
    \end{align*}
    as required.
\end{proof}
This particular construction tackles the problem using a na\"ive, ``brute force" approach and the size of the cubature formula, $2^d$, grows exponentially with dimension $d$. However, this cubature formula satisfies any degree $m$, making it an exact approximation. In many situations a cubature formula of a certain given degree may suffice. Naturally, in these instances, we can develop cubature formulae with far fewer points. We will refer to a cubature formula as \textit{efficient} if there exists some $k$ such that the size of the cubature formula, $n$, satisfies,
\[
    n\leq Cd^{k},
\]
for a fixed constant $C$ and all values $d$. The existence of efficient quadrature formulae is guaranteed by Tchakaloff's theorem \cite{tchakaloff1957formules}, which was later generalised to the case of cubature by Putinar \cite{PutinarMihai1997AnoT}.

\subsection{Wiener space}
Wiener space can be viewed as the probability space generated by a \textit{Brownian motion}, a random process with Gaussian increments. Formally, this is given by the following definitions.
\begin{definition}[Brownian motion]
    Let $B$ be a stochastic process on $(\Omega,\mathcal{F},\mathbb{P})$, with values in $\mathbb{R}^d$ such that,
    \begin{enumerate}[(I)]
        \item For any $0\leq t_0\leq \cdots\leq t_n$, the random variables,
        \[
        B_{t_0},B_{t_1}-B_{t_0},\cdots,B_{t_n}-B_{t_{n-1}},
        \]
        are independent;
        \item For any $0\leq s<t$, the random variable $B_t-B_s$ has multivariate normal distribution $\mathcal{N}(0,t-s)$;
        \item Almost all sample paths of $B$ are continuous.
    \end{enumerate}
    Then $B$ is a \textit{Brownian motion} in $\mathbb{R}^d$. Moreover, if $B_0=0$, then $B$ is a \textit{standard Brownian motion}.
\end{definition}
The existence and uniqueness theorems for Brownian motion are non-trivial and an explicit construction was first proposed by Wiener \cite{WienerNorbert1938THC}. In this dissertation, it is helpful to consider an augmented Brownian motion, $\widehat{B}_t$, a process in $\mathbb{R}^{d+1}$ where the extra dimension is time, $\widehat{B}_t^0:=t$.
\begin{definition}[Wiener measure]
    Let $B$ be a standard Brownian motion in $\mathbb{R}^{d}$. Then the law of process $B$ is the Wiener measure on $\mathbb{R}^{d}$.
\end{definition}
It is useful to remark that the Wiener measure defined in this way is a Gaussian measure.
\begin{definition}[Wiener space]
    Let $(\Omega, \mathcal{F}, \mathbb{P})$ be a probability space where,
    \begin{enumerate}
        \item $\Omega = C_0^0([0,T],\mathbb{R}^d)$, the space of $\mathbb{R}^d$-valued continuous functions defined in $[0,T]$ which start at zero;
        \item $\mathcal{F}$ is the associated Borel $\sigma$-algebra;
        \item $\mathbb{P}$ is the Wiener measure on $\mathbb{R}^{d}$.
    \end{enumerate}
    Then $(\Omega, \mathcal{F}, \mathbb{P})$ is the Wiener space in $\mathbb{R}^{d}$ .
\end{definition}
Heuristically, Wiener space is the distribution of the possible sample paths traced by a standard Brownian motion. It is typical to refer to paths in their time-augmented sense, and so the elements of Wiener space assume values in $\mathbb{R}^{d+1}$. Inherited from the properties of Brownian motion is that paths are continuous, but no further smoothness conditions are guaranteed, since Brownian motion is not differentiable.
\section{Cubature Formula on Wiener Space}\label{chap:cub-def}
Giving a definition to cubature formulae for Wiener space is not as simple as for positive measures on $\mathbb{R}^d$. Elements of Wiener space are sample paths, with some probabilistic distribution. To continue, we must enter a realm of stochastic analysis known as \textit{rough path theory}. In particular, we will develop the concept of the \textit{path-signature}, an enhancement to the regular notion of path that will be the driving motivation behind our definition of a cubature formula on Wiener space.

 Rough path theory relies upon a foundational knowledge of tensor algebras and lie algebras. The following section will give the relevant details from these respective areas, which we will subsequently utilise to define path-signatures and cubature formulae on Wiener space.

\subsection{Tensor algebra and the free Lie algebra}\label{sec:algebras}
\subsubsection{Alphabets, words and Lyndon words}
We begin with a totally ordered vector space $W$, often referred to as the \textit{alphabet}. We refer to tensors $w\in W^{\otimes k}$ as a \textit{word} of length $k$. We denote the set of all words by $W^*=\bigcup_{k=1}^\infty W^{\otimes k}$. $W^*$ inherits a natural ordering from the total order of the alphabet, known as \textit{lexicographical ordering}. We define this ordering by the following: suppose $w^{(1)}\in W^{\otimes k_1}, w^{(2)}\in W^{\otimes k_2}$ are two distinct words of length $k_1$ and $k_2$ respectively. Moreover, without loss of generality, let $k_1\leq k_2$. Then,
\begin{enumerate}[(I)]
    \item if there exists a word $w^{(3)}$ such that $w^{(1)}=w^{(2)}\otimes w^{(3)}$, then $w^{(1)}<w^{(2)}$;
    \item otherwise, write $w^{(1)} = w^{(1)}_1\otimes \cdots \otimes w^{(1)}_{k_1}$ and $w^{(2)}_1\otimes \cdots \otimes w^{(2)}_{k_2}$. We say $w^{(1)}<w^{(2)}$ if and only if $w_i^{(1)}<w_i^{(2)}$ where $i=\min\{j\leq k_1 : w_j^{(1)}\neq w_j^{(2)}\}$.
\end{enumerate}
The lexicographical ordering reflects the ordering of any standard English dictionary, hence the use of the terms \textit{word} and \textit{alphabet}.

\begin{definition}[Lyndon word]
    A word, $w=w_1\otimes\cdots\otimes w_k$, is a Lyndon word if it is lexicographically smaller than all cyclic rearrangements of itself. That is, $w<w^*$ for all $w^*\in\{w_{\sigma(1)}\otimes\cdots\otimes w_{\sigma(k)}\ :\ \sigma\in Z_k\}$, where $Z_k$ denotes the cyclic group of order $k$.
\end{definition}
This definition is attributed to Lyndon \cite{LyndonR.C.1954OBP}.

\subsubsection{Tensor algebra}
The tensor algebra over $W$ is the set of linear combinations of words in $W^*$,
\[
T(W) = \bigoplus_{k=0}^\infty W^{\otimes k}.
\]
The $m$-truncated tensor algebra over $W$ is given by,
\[
T^{(m)}(W) =  \bigoplus_{k=0}^m W^{\otimes k}.
\]
We define the projection, $\pi_m : T(W)\rightarrow T^{(m)}$(W), to be the map that retains only tensors of order less than or equal to $m$.

\subsubsection{Lie algebra and the free Lie algebra}
A lie bracket on $W$, is an alternating bilinear map, $[\cdot,\cdot]:W\times W\rightarrow W$, that satisfies the Jacobi identity,
\[
\big [u,[v,w]\big]+\big [v,[w,u]\big]+\big [w,[u,v]\big]=0,
\]
for all $u,v,w\in W$. We define the set of \textit{iterated Lie brackets}, $\mathcal{M}(W)$, (also called the set of \textit{Lie monomials}) by the following recursion,
\begin{enumerate}[(I)]
    \item $w\in\mathcal{M}(W)$, for all $w\in W$;
    \item If $p,q\in\mathcal{M}(W)$ then $[p,q]\in\mathcal{M}(W)$.
\end{enumerate}
For each iterated Lie bracket, we can associate a word by reading the entries in order from left to right. Many iterated Lie brackets have the same associated word. For instance,
\[
    \big [[u,v],[u,v]\big ]\quad\text{and}\quad \Big [\big[[u,v],u\big ],v\Big ],
\]
both have associated word $u\otimes v\otimes u\otimes v$. This allows us to instill $\mathcal{M}(W)$ with the following (weak) ordering: given $\ell_1,\ell_2\in\mathcal{M}(W)$, with associated words $w_1,w_2$ respectively, we say $\ell_1<\ell_2$ only if $w_1<w_2$.

This ordering is weak due to the lack of injectivity between Lie brackets and their associated words. 

\begin{definition}[Free Lie algebra]
    The space of linear combinations of iterated Lie brackets, $\mathcal{L}(W)$, is called the \textit{free Lie algebra} generated by $W$.
\end{definition}
Finite combinations of iterated lie brackets will be referred to as \textit{Lie polynomials}, while infinite combinations will be know as \textit{Lie series}. The weak ordering on $\mathcal{M}(W)$ extends to $\mathcal{L}(W)$ by linearity. 

For the purposes of this work, we will refer to the \textit{Lyndon basis} for the free lie algebra, an instance of a larger class of bases known as the \textit{Hall bases}. The Lyndon basis is defined using the Lyndon words as an indexing set for iterated Lie brackets, the full details of which are given by Reutenauer \cite{alma990107836220107026}.

Finally, we define a symmetrised product on the free Lie algebra.
\begin{definition}[Symmetrised product]
    Given lie polynomials $\ell_1,\cdots,\ell_n\in\mathcal{L}(W)$, their \textit{symmetrised product} is defined by,
    \[
    (\ell_1,\cdots,\ell_n):= \frac{1}{n!}\sum_{\sigma\in S_n}\ell_{\sigma(1)}\otimes\cdots\otimes\ell_{\sigma(n)},
    \]
    where $S_n$ denote the symmetric group of degree $n$.
\end{definition}
This definition motivates an important corollary of the \textit{Poincar\'e-Birkhoff-Witt theorem}, one that will give a basis for the tensor algebra in terms of Lie polyomials.
\begin{theorem}[Poincar\'e-Birkhoff-Witt theorem]\label{thm:PBW}
    Let $\mathcal{B}_\mathcal{L}$ denote a basis of the free lie algebra generated by $W$. Then the set,
    \[
    \mathcal{B}_{T(W)}:=\bigcup_{n\geq 0}\Big \{(\ell_1,\cdots,\ell_n)\quad :\quad\begin{array}{l}
    \ell_1,\cdots,\ell_n\in\mathcal{B}_\mathcal{L},\\
    \ell_1\leq \cdots\leq \ell_n
    \end{array}\Big \},
    \]
    forms a basis for the tensor algebra, $T(W)$.
\end{theorem}
When $\mathcal{B}_\mathcal{L}$ denotes the Lyndon basis, we will refer to the basis for $T(W)$ as the \textit{symmetrised Lyndon basis}.

The exponential of a Lie polynomial, $\ell$, is defined by,
\[
    \exp (\ell) := \sum_{k=0}^\infty \frac{1}{k!}\ell^{\otimes k}.
\]
A useful result is that if we express a Lie polynomial in the Lyndon basis, then the exponential of the Lie polynomial can be easily written in the symmetrised Lyndon basis.

\subsection{Paths, signatures and log-signatures}\label{sec:path-sigs}
The signature is an enhancement of a path that captures feature information, first introduced by Lyons \cite{alma991025219949807026}. It serves to equip a path with the relevant ``additional" information that encapsulates the roughness of the path in a way that ensures solution maps to differential equations are continuous.
\begin{definition}[Signature of a path]
    Let $X_t$ be a path in $\mathbb{R}^d$. For any multi-index $I=(i_1,\cdots,i_k)$, define the iterated integral associated with $w$ by,
    \[    X_{[s,t]}^{I}=\int_{s<t_1<\cdots<t_k<t}\text{d}X^{i_1}_{t_1}\otimes\cdots\otimes\text{d}X^{i_k}_{t_k}.
    \]
    The \textit{signature} of $X_t$ over the interval $[s,t]$, is defined by the collection of all such integrals, often written as an element of the tensor algebra over $\mathbb{R}^d$,
    \[
        S_{[s,t]}(X):= \sum_{I\in I^*}X_{[s,t]}^I\in T(\mathbb{R}^d),
    \]
    where $I^*$  denotes the set of all multi-indices.
\end{definition}
The path-signature is an infinite series, so in practical settings we often work with the the depth-$m$ truncated signature given by,
\[
    S^{(m)}_{[s,t]}(X):=\pi_m \big(S_{[s,t]}(X)\big)\in T^{(m)}(\mathbb{R}^d).
\]
The path-signature contains many redundancies and has a more efficient representation, called the \textit{log-signature}.
\begin{definition}[Log-signature]
    The \textit{log-signature} is defined using the logarithm of a formal series,
    \[
        \log\big(S_{[s,t]}(X)\big ):=\sum_{k=1}^\infty \frac{(-1)^k}{k}\big (1-S_{[s,t]}(X)\big )^{\otimes k}.
    \]
\end{definition}

\subsection{Cubature formula on Wiener space definitions}\label{sec:cub-def}
Recall that a cubature formula for a positive measure was defined using an equivalence relation on polynomials up to degree $m$, since polynomials can be used to approximate functions to within some accuracy that depends on $m$. The analogous tool for cubature formulae on Wiener space turns out to be path-signatures up to truncation $m$, an idea first defined by Lyons and Victoir \cite{LyonsTerry2004CoWs}. The more levels to which two path-signatures agree, the better the underlying paths are to ``approximating" one another. 

\subsubsection{Cubature in terms of paths}
The simplest way to define cubature is as a discrete set of weights and paths that weakly approximate the signature of Wiener space.
\begin{definition}[Cubature Formula on Wiener Space]
    A set of paths, $\omega_1,\cdots,\omega_n\in C_0^0([0,T],\mathbb{R}^d)$ and associated positive weights $\lambda_1,\cdots,\lambda_n$ define a degree $m$ cubature formula on Wiener space in $\mathbb{R}^d$ at time $T$ if,
    \[
        \mathbb{E}\big [S^{(m)}_{[0,T]}(\circ \widehat{B})\big ]=\sum_{k=1}^n\lambda_kS_{[0,T]}^{(m)}(\omega_k).
    \]
\end{definition}
The notation $\circ$ is used to indicate integration understood in the sense of Stratonovich \cite{alma990107465930107026}, a convention which will be maintained throughout.

\subsubsection{Cubature in terms of Lie polynomials}
Theorem \ref{thm:PBW} yielded an alternative basis for $T(\mathbb{R}^d)$ in terms of Lie polynomials. This is indicative of a mapping between the tensor alegbra and the free Lie alegbra, which can be applied to the path-signature. The details of this are captured in Chen's theorem \cite{ChenKuo-sai1958Iopf}.
\begin{theorem}[Chen's theorem]
    Let $X_t$ be a path in $\mathbb{R}^d$. Then the path-signature, $S_{[s,t]}(X)$, has the following multiplicative property,
    \[
    S_{[r,s]}(X)\otimes S_{[s,t]}(X) = S_{[r,t]}(X).
    \]
    Furthermore $\log\big(S_{[s,t]}(X)\big )$ is a Lie series. Conversely, suppose $\ell\in\mathcal{L}^{(m)}(\mathbb{R}^d)$ is a Lie polynomial. There exists a continuous path of bounded variation, $\omega$, such that,
    \[
    \log\big (S^{(m)}_{[s,t]}(\omega)\big )=\ell.
    \]
\end{theorem}
This motivates an equivalent definition for cubature formulae in Wiener space, one that uses Lie polynomials instead of paths.
\begin{definition}[Cubature formula on Wiener space]\label{def:cub-lie-poly}
    A set of Lie polynomials, $\ell_1,\cdots,\ell_n\in\mathcal{L}^{(m)}(\mathbb{R}^{d+1})$ and associated positive weights $\lambda_1,\cdots,\lambda_n\in\mathbb{R}$ define a degree $m$ cubature formula on Wiener space in $\mathbb{R}^d$ if,
    \[
        \mathbb{E}\big [S_{[0,1]}^{(m)}(\circ B)\big ]=\sum_{k=1}^n\lambda_k\pi_m\big (\exp(\ell_k)\big ).
    \]
\end{definition}
It transpires that this alternative definition aids the explicit construction of cubature formulae, as will be seen in Chapter \ref{chap:construct}. To facilitate this is the following expansion, stated as Proposition 4.10 in Lyons and Victoir \cite{LyonsTerry2004CoWs}.
\[
    \mathbb{E}\big [S_{[0,1]}(\circ B)\big ]=\exp \Big (\varepsilon_0+\frac{1}{2}\sum_{k=1}^d\varepsilon_k^{\otimes 2}\Big ).
\]
Substituting this expression into Definition \ref{def:cub-lie-poly}, then a set of Lie polynomials and associated weights which satisfy,
\[
    \pi_m\Bigg (\exp\Big (\epsilon_0 +\frac{1}{2}\sum_{k=1}^d \epsilon_k^{\otimes 2}\Big )\Bigg ) = \sum_{k=1}^n\lambda_k\pi_m\big (\exp(\ell_k)\big ),
\]
is a cubature formula of degree $m$.
\section{Explicit Constructions of Cubature Formula on Wiener Space}\label{chap:construct}
The first explicit constructions for efficient cubature formula on Wiener space were published in Lyons and Victoir \cite{LyonsTerry2004CoWs}, which contains constructions for degree 3 and degree 5 cubature formula for Wiener space in general dimension. Various constructions for a degree 7 cubature formula have been proposed. Most notably, Litterer \cite{LittererThesis} proposed an efficient degree 7 cubature formula for Wiener space in dimension 2 and Gyurk\'o and Lyons \cite{alma991025225271007026} published a selection of constructions for low-dimensional cases and in various degrees. Nonetheless, an explicit degree 7 construction in general dimension has remained elusive.

In this section, we will review the degree 3 and degree 5 constructions laid out in Lyons and Victoir \cite{LyonsTerry2004CoWs}. These constructions motivate a general framework, which will then be applied to the degree 7 case. The section will conclude with an explicit construction of a degree 7 cubature formula for dimension 3 Wiener space, and a discussion on how this is sufficient to define a formula in general dimension.

Recall from the previous section that to obtain a degree $m$ cubature formula on Wiener space, it is sufficient to construct Lie polynomials $\ell_1,\cdots,\ell_n\in\mathcal{L}^{(m)}(\mathbb{R}^{d+1})$ and associated weights $\lambda_1,\cdots,\lambda_n$ which satisfy,
\begin{equation}\label{eq:cub}
    \pi_m\Bigg (\exp\Big (\epsilon_0 +\frac{1}{2}\sum_{k=1}^d \epsilon_k^{\otimes 2}\Big )\Bigg ) = \sum_{k=1}^n\lambda_k\pi_m\big (\exp(\ell_k)\big ).
\end{equation}

\subsection{Degree $m=3$}

\begin{theorem}
    Let $(z_k,\lambda_k)_{k=1}^n$ be a degree 3 cubature formula for the $d$-dimensional Gaussian measure. Define,
    \[
        \ell_k = \varepsilon_0+\sum_{i=1}^d z_k^i\varepsilon_i,
    \]
    for $k=1,\cdots,n$. Then $(\ell_k,\lambda_k)_{k=1}^n$ defines a degree 3 cubature formula on Wiener space.
\end{theorem}
\begin{proof}
Expressed in the symmetrised Lyndon basis, the left-hand side of (\ref{eq:cub}) becomes,
\[
    \pi_3\Bigg (\exp\Big (\epsilon_0 +\frac{1}{2}\sum_{i=1}^d \epsilon_i^{\otimes 2}\Big )\Bigg ) = 1+\varepsilon_0+\frac{1}{2}\sum_{i=1}^d (\varepsilon_i,\varepsilon_i)
\]
and we have the following expansion,
\[
    \pi_3\exp(\ell_k) = 1+\varepsilon_0+\sum_{i=1}^d z_k^i\varepsilon_k + \frac{1}{2}\sum_{i=1}^d (z_k^i)^2(\varepsilon_i,\varepsilon_i)+\sum_{1\leq i<j\leq d} z_k^iz_k^j(\varepsilon_i,\varepsilon_j).
\]
Let $Z=(Z_1,\cdots,Z_d)$ be a $d$-dimensional Gaussian variable. Using the moments of the Gaussian,
\begin{align*}
    &\sum_{k=1}^n \lambda_kz_k^i=\mathbb{E}[Z_i]=0\\
    &\sum_{k=1}^n \lambda_kz_k^iz_k^j=\mathbb{E}[Z_iZ_j]=0\\
    &\sum_{k=1}^n \lambda_k(z_k^i)^2=\mathbb{E}[Z_i^2]=1
\end{align*}
for all $i<j$. Thus,
\[
    \sum_{k=1}^n\lambda_k\pi_m\exp(\ell_k) = 1+\varepsilon_0+\frac{1}{2}\sum_{i=1}^d (\varepsilon_i,\varepsilon_i),
\]
as required.
\end{proof}
Given the prevalence of the (multivariate) Gaussian measure in probability, statistics and other fields, there is a history of literature that presents efficient cubature formulae for the Gaussian measure. The encyclopaedic work presented by Stroud \cite{alma990129986400107026} contains many such cubature formulae. Wiener space is intrinsically connected to the Gaussian measure, so it is no coincidence that we can utilise it. If every instance of a canonical basis variable, $\varepsilon_i$, appears with a coefficient of $z_i$, then the moments of the Gaussian distribution ensure that whatever basis terms remain after the right-hand side expansion of (\ref{eq:cub}) only contains even occurrences of $\varepsilon_i$. This simple trick will be employed again to construct a degree 5 cubature formula.

The size of the cubature formula is given by,
\[
\mathcal{S}_d(3)=N_d(3)
\]
where $N_d(3)$ is the size of the $d$-dimensional Gaussian cubature formula of degree 3.

\subsection{Degree $m=5$}
\begin{theorem}
    Let $(z_k,\lambda_k)_{k=1}^n$ be a degree 5 cubature formula for the $d$-dimensional Gaussian measure. Define,
    \begin{align}
        \ell_{k,\eta} &= \varepsilon_0+\sum_{i=1}^d z_k^i\varepsilon_i+\frac{1}{12}\sum_{i=1}^d (z_k^i)^2\big [[\varepsilon_0,\varepsilon_i],\varepsilon_i\big ]+\frac{\eta}{2}\sum_{1\leq i<j\leq d}z_k^iz_k^j[\varepsilon_i,\varepsilon_j]\\
        &+\frac{1}{6}\sum_{1\leq i<j\leq d}\bigg ( xz_k^i(z_k^j)^2\big[[\varepsilon_i,\varepsilon_j],\varepsilon_j\big ]-(1-x) (z_k^i)^2z_k^j\big[[\varepsilon_i,\varepsilon_j],\varepsilon_i\big ]\bigg )
    \end{align}
    for $k=1,\cdots,n$ and $\eta\in \{1,-1\}$. Further, let $\lambda_{k,\eta}=\frac{\lambda_k}{2}$. Then $(\ell_{k,\eta},\lambda_{k,\eta})_{k,\eta}$ defines a degree 5 cubature formula on Wiener space.
\end{theorem}
\begin{proof}
Let $Z=(Z_1,\cdots,Z_d)$ be a $d$-dimensional Gaussian variable and let $\Lambda$ be a Bernoulli variable. The following table contains the left-hand side and right-hand side expansions of (\ref{eq:cub}) as coefficients of the symmetrised Lyndon basis. For any $1\leq i<j\leq d$,
\begin{center}
\bgroup
\def\arraystretch{1.5}
\begin{tabular}{|c|c|c|} 
     \hline
     Basis Element & LHS & RHS\\
     \hline
     $\varepsilon_0$ & $\frac{1}{2}$ & $\frac{1}{2}$\\
     $(\varepsilon_0,\varepsilon_0)$ & $\frac{1}{6}$ & $\frac{1}{6}$\\
     $(\varepsilon_i,\varepsilon_i)$ & $\frac{1}{2}$ & $\frac{1}{2}\mathbb{E}[Z_i^2]$\\
     $(\varepsilon_0,\varepsilon_i,\varepsilon_i)$ & $\frac{1}{2}$ & $\frac{1}{2}\mathbb{E}[Z_i^2]$\\
     $\big [[\varepsilon_0,\varepsilon_i],\varepsilon_i\big ]$ & $\frac{1}{12}$ & $\frac{1}{12}$\\[1ex]
     \hline
\end{tabular}
\quad\quad
\begin{tabular}{|c|c|c|}
     \hline
     Basis Element & LHS & RHS\\
     \hline
     $(\varepsilon_i,\varepsilon_i,\varepsilon_i,\varepsilon_i)$ & $\frac{1}{8}$ & $\frac{1}{24}\mathbb{E}[Z_i^4]$\\
     $(\varepsilon_i,\varepsilon_i,\varepsilon_j,\varepsilon_j)$ & $\frac{1}{4}$ & $\frac{1}{4}\mathbb{E}[Z_i^2Z_j^2]$\\
     $([\varepsilon_i,\varepsilon_j],[\varepsilon_i,\varepsilon_j])$ & $\frac{1}{8}$ & $\frac{1}{8}\mathbb{E}[\Lambda^2Z_i^2Z_j^2]$\\
     $(\big [[\varepsilon_i,\varepsilon_j],\varepsilon_j\big ],\varepsilon_i)$ & $\frac{1}{12}$ & $\frac{1}{12}\mathbb{E}[Z_i^2Z_j^2]$\\
     $(\big [\varepsilon_i,[\varepsilon_i,\varepsilon_j]\big ],\varepsilon_j)$ & $\frac{1}{12}$ & $\frac{1}{12}\mathbb{E}[Z_i^2Z_j^2]$\\[1ex]
     \hline
\end{tabular}
\egroup
\end{center}
All other basis terms have coefficient zero, due to either the Gaussian or Bernoulli moments. Each right-hand side coefficient matches the corresponding left-hand side coefficient, so (\ref{eq:cub}) is satisfied.
\end{proof}
This construction required the use of a cubature formula for both the Gaussian measure and for the Bernoulli measure. The Gaussian measure alone, whilst it removes a significant portion of undesirable terms, does not fulfil the task entirely. To remove excess terms in the right-hand side expansion is to decorrelate their constituent factors. This is easily achieved by introducing a Bernoulli variable.

The size of the cubature formula is given by,
\[
\mathcal{S}_d(5)=2N_d(5),
\]
where $N_d(5)$ is the size of the $d$-dimensional Gaussian cubature formula of degree 5.

\subsection{Degree $m=7$}
Using a degree 7 cubature formula for the Gaussian measure, combined with a number of Bernoulli variables allows us to present a degree 7 cubature formula for dimension 3 Wiener space.

\begin{theorem}
Let $(z_k,\lambda_k)_{k=1}^{n_1}$ be a degree 7 cubature formula for the $3$-dimensional Gaussian measure. Let $(\eta_k,\mu_k)_{k=1}^{n_2}$ be a degree 5 cubature formula for the four dimensional Bernoulli measure. Define,
\begin{align*}
    \ell_{r,s} &= \varepsilon_0+\sum_{i=1}^d z_r^i\varepsilon_i+\frac{1}{6}\eta_s^1\eta_s^2\Big (z_r^1\big [[\varepsilon_1,\varepsilon_2],\varepsilon_2\big]+z_r^2\big [[\varepsilon_2,\varepsilon_3],\varepsilon_3\big]+z_r^3\big [[\varepsilon_3,\varepsilon_1],\varepsilon_1\big]\Big )\\
    &+\frac{1}{6}\eta_s^1\eta_s^3\Big (z_r^1\big [[\varepsilon_1,\varepsilon_3],\varepsilon_3\big]+z_r^2\big [[\varepsilon_2,\varepsilon_1],\varepsilon_1\big]+z_r^3\big [[\varepsilon_3,\varepsilon_2],\varepsilon_2\big]\Big )\\
    &+\frac{1}{6}\eta_s^2\eta_s^3\Big (\big [[\varepsilon_1,\varepsilon_2],\varepsilon_3\big]+\big [[\varepsilon_2,\varepsilon_3],\varepsilon_1\big]+\big [[\varepsilon_3,\varepsilon_1],\varepsilon_2\big]\Big )+\frac{1}{12}\sum_{i=1}^d\big [[\varepsilon_0,\varepsilon_i],\varepsilon_i\big]\\
    &+\frac{1}{2\sqrt{3}}\Big (\eta_s^1\eta_s^2\eta_s^0-\eta_s^2z_r^1+\eta_s^1z_r^2\Big )[\varepsilon_1,\varepsilon_2]+\frac{1}{2\sqrt{3}}\Big (\eta_s^1\eta_s^3\eta_s^0+\eta_s^3z_r^1+\eta_s^1z_r^3\Big )[\varepsilon_1,\varepsilon_3]\\
    &+\frac{1}{2\sqrt{3}}\Big (\eta_s^2\eta_s^3\eta_s^0+\eta_s^3z_r^2+\eta_s^2z_r^3\Big )[\varepsilon_2,\varepsilon_3]+\frac{1}{2\sqrt{3}}\Big (\eta_s^3[\varepsilon_0,\varepsilon_3]-\eta_s^2[\varepsilon_0,\varepsilon_2]-\eta_s^1[\varepsilon_0,\varepsilon_1]\Big )\\
    &+\frac{1}{24\sqrt{3}}\eta_s^0\sum_{\substack{1\leq i,j,k\leq d\\i<j\\i\neq k\neq j}} \eta_s^i\eta_s^j\Big (\Big [\big [[\varepsilon_k,\varepsilon_j],\varepsilon_k\big],\varepsilon_i\Big]+\Big [\big [[\varepsilon_i,\varepsilon_j],\varepsilon_k\big],\varepsilon_k\Big]+\Big [\big [[\varepsilon_i,\varepsilon_k],\varepsilon_k\big],\varepsilon_j\Big]\Big )\\
    &+\frac{1}{12\sqrt{3}}\eta_s^0\sum_{\substack{1\leq i,j,k\leq d\\i<j}} \eta_s^i\eta_s^j\Big (\Big [\big [[\varepsilon_i,\varepsilon_j],\varepsilon_j\big],\varepsilon_j\Big]+\Big [\big [[\varepsilon_i,\varepsilon_j],\varepsilon_i\big],\varepsilon_i\Big]\Big )+\frac{1}{12}\sum_{\substack{i,j\leq d \\ i\neq j}}\big [[\varepsilon_i,\varepsilon_j],\varepsilon_j\big]
\end{align*}
\begin{align*}
    &+\frac{1}{360}\sum_{i<j<k}\Big (z_r^i\bigg [\Big [\big [[\varepsilon_j,\varepsilon_k],\varepsilon_k\big],\varepsilon_j\Big],\varepsilon_i\bigg ]+z_r^k\bigg [\Big [\big [[\varepsilon_i,\varepsilon_j],\varepsilon_j\big],\varepsilon_k\Big],\varepsilon_i\bigg ]+z_r^i\bigg [\Big [\big [[\varepsilon_k,\varepsilon_i],\varepsilon_j\big],\varepsilon_j\Big],\varepsilon_k\bigg ]\\
    &+z_r^k\bigg [\Big [\big [[\varepsilon_j,\varepsilon_i],\varepsilon_i\big],\varepsilon_k\Big],\varepsilon_j\bigg ]+z_r^j\bigg [\Big [\big [[\varepsilon_k,\varepsilon_j],\varepsilon_i\big],\varepsilon_i\Big],\varepsilon_k\bigg ]+z_r^j\bigg [\Big [\big [[\varepsilon_i,\varepsilon_k],\varepsilon_k\big],\varepsilon_i\Big],\varepsilon_j\bigg ]\Big )\\
    &+\frac{1}{180}\sum_{i<j<k}\Big (z_r^j\bigg [\Big [\big [[\varepsilon_k,\varepsilon_i],\varepsilon_i\big],\varepsilon_j\Big],\varepsilon_k\bigg ]+z_r^i\bigg [\Big [\big [[\varepsilon_k,\varepsilon_j],\varepsilon_j\big],\varepsilon_i\Big],\varepsilon_k\bigg ]+z_r^k\bigg [\Big [\big [[\varepsilon_j,\varepsilon_k],\varepsilon_i\big],\varepsilon_i\Big],\varepsilon_j\bigg ]\\
    &+z_r^k\bigg [\Big [\big [[\varepsilon_i,\varepsilon_k],\varepsilon_j\big],\varepsilon_j\Big],\varepsilon_i\bigg ]\Big )+\frac{1}{120}\sum_{i<j<k}\Big (z_r^i\bigg [\Big [\big [[\varepsilon_j,\varepsilon_i],\varepsilon_k\big],\varepsilon_k\Big],\varepsilon_j\bigg ]+z_r^j\bigg [\Big [\big [[\varepsilon_i,\varepsilon_j],\varepsilon_k\big],\varepsilon_k\Big],\varepsilon_i\bigg ]\Big )\\
    &+\frac{1}{120}\sum_{\substack{1\leq i,j,k\leq d\\i\neq j\neq k\\
    i\neq k}}z_r^i\bigg [\Big [\big [[\varepsilon_i,\varepsilon_k],\varepsilon_k\big],\varepsilon_j\Big],\varepsilon_j\bigg ]+\frac{1}{360}\sum_{\substack{1\leq i,j\leq d\\i\neq j}}z_r^i\bigg [\Big [\big [[\varepsilon_i,\varepsilon_j],\varepsilon_j\big],\varepsilon_j\Big],\varepsilon_j\bigg ]\\
    &+\frac{1}{120}\sum_{\substack{1\leq i,j\leq d\\i\neq j}}z_r^j\bigg [\Big [\big [[\varepsilon_i,\varepsilon_j],\varepsilon_j\big],\varepsilon_j\Big],\varepsilon_i\bigg ]+\frac{1}{90}\sum_{\substack{1\leq i,j\leq d\\i\neq j}}z_r^i\bigg [\Big [\big [[\varepsilon_i,\varepsilon_j],\varepsilon_i\big],\varepsilon_j\Big],\varepsilon_i\bigg ]\\
    &+\frac{1}{360}\sum_{i=1}^d\bigg [\Big [\big [[\varepsilon_0,\varepsilon_i],\varepsilon_i\big],\varepsilon_i\Big],\varepsilon_i\bigg ]+\frac{1}{120}\sum_{\substack{1\leq i,j\leq d\\i\neq j}}\bigg [\Big [\big [[\varepsilon_0,\varepsilon_i],\varepsilon_i\big],\varepsilon_j\Big],\varepsilon_j\bigg ]\\
    &+\frac{1}{180}\sum_{\substack{1\leq i,j\leq d\\i\neq j}}\bigg [\Big [\big [[\varepsilon_j,\varepsilon_0],\varepsilon_i\big],\varepsilon_i\Big],\varepsilon_j\bigg ]+\frac{1}{360}\sum_{\substack{1\leq i,j\leq d\\i\neq j}}\bigg [\Big [\big [[\varepsilon_i,\varepsilon_j],\varepsilon_j\big],\varepsilon_0\Big],\varepsilon_i\bigg ]
\end{align*}
Further define $\lambda_{r,s}=\lambda_{r}\mu_{s}$. Then $\big \{\big (\ell_{r,s}, \lambda_{r,s}\big )\ :\ r\leq n_1,\ s\leq n_2\big \}$ defines a degree 7 cubature formula on Wiener space in $\mathbb{R}^3$.
\end{theorem}
\begin{proof}
    The number of basis terms that have degree up to 7 is inconveniently large. The left-hand side expansion of (\ref{eq:cub}) for $m=7$ can be found in Appendix \ref{app:expansion}. The construction is most easily verified using the Python code that computes the right-hand side expansion found in Appendix \ref{app:code}.
\end{proof}

The size of the cubature formula is given by,
\[
\mathcal{S}_3(7)=N_3(7)B_4(5),
\]
where $N_3(7)$ and $B_4(5)$ are the sizes of the three-dimensional Gaussian cubature formula and four-dimensional Bernoulli cubature formula respectively.

Importantly, when $m=7$, the left-hand side expansion of (\ref{eq:cub}) contains terms which are composed of at most three distinct canonical basis variables. Consequently a symmetric construction for dimension 3 Wiener space is sufficient for the general dimension case. A symmetric formula will include terms from a larger spanning set than the symmetrised Lyndon basis, due to the asymmetric nature of the Lyndon basis. However, the construction above contains some significant asymmetries, particularly involving the single lie bracket terms.
\section{Uses of Cubature Formula on Wiener Space}\label{chap:uses}
Consider the (multivariate) controlled differential equation (understood in the sense of Stratonovich),
\begin{equation}\label{eq:CDE}
    dX_t = \mu (X_t,t)\text{d}t+\sigma (X_t,t)\circ\text{d}B_t,
\end{equation}
coupled with some initial condition for $X_0$. The path $X_t$ is governed by a Brownian motion, subject to some drift function $\mu$ and some diffusion function $\sigma$. We are interested in computing the expected value of $\phi(X_T)$ for some $T>0$, where $\phi$ is a given bounded, Lipschitz function. Chapter \ref{chap:deriv-pricing} of this paper will explore financial market models which adopt the form (\ref{eq:CDE}).

In the absence of analytical solutions to (\ref{eq:CDE}), the widely accepted ``na\"ive" method used to approximate $\mathbb{E}[\phi(X_T)]$ is \textit{Monte Carlo sampling} \cite{MetropolisNicholas1949TMCM}. This method requires some fixed number of sample paths, and a partition of the domain $[0,T]$ into some number of intervals.  The Monte Carlo method propagates each sample path in accordance with the controlled differential equation over each interval in the partition. The resulting set of randomly-generated sample paths is then used to compute the expected value of $\phi(X_T)$. 

Naturally, the computational cost of this process is high. Since the paths are generated at random from the (infinite) set of possible paths, we require a large sample size to obtain any reliable approximations. The rough nature of Brownian motion means we also require a large number of intervals. The motivation behind cubature is to give structure to the selection of such paths and produce a method that requires far fewer sample paths or intervals, without altering the success of the approximation.

This chapter will suggest two alternative methods to Monte Carlo sampling that utilise cubature formulae. Namely, we introduce an adaptation of Taylor's theorem that applies to stochastic processes, followed by a more sophisticated approach known as the log-ODE method. The chapter will conclude with a final section discussing the expected path-signature, and how cubature methods may pave a way towards a quantifiable approximation for it.

\subsection{Taylor method}
If we rewrite (\ref{eq:CDE}) using an augmented Brownian motion, we can combine the drift and diffusion functions into a single linear map $f:\mathbb{R}^n\rightarrow \mathbb{R}^{d+1}$. The differential equation becomes,
\begin{align}\label{eq:CDE-aug}
    \text{d}X_t = f(X_t)\circ\text{d}\widehat{B}_t.
\end{align}
We will make use of vector field notation to represent the derivatives of $f$. That is, for each value of $x\in\mathbb{R}^n$, we will define a series of functions that describes the directional derivatives in direction $x$.
\begin{definition}[Derivatives as a vector field]
    Define $f^{\circ k}:\mathbb{R}^n\rightarrow L \big ((\mathbb{R}^n)^{\otimes k},\mathbb{R}^d\big )$ recursively by,
    \begin{enumerate}[(I)]
        \item $f^{\circ 0}(x)=x$;
        \item $f^{\circ 1}(x)=f(x)$;
        \item $f^{\circ k}(x)=D\big (f^{\circ k-1}\big )(x)f(x)$;
    \end{enumerate}
    for $k\geq 2$, $x\in\mathbb{R}^d$ and where $D(f^{\circ k})$ denotes the Fr\'echet derivative of $f^{\circ k}$.    
\end{definition}
Viewing derivatives in this way sheds light on a version of Taylor's theorem that applies to stochastic differential equations. It is stated here in the context of Wiener space.
\begin{theorem}[Taylor's theorem in Wiener space]\label{thm:taylor}
    Let $\phi : \mathbb{R}^n\rightarrow \mathbb{R}$ be a smooth, bounded function. Let $\mathcal{S}_{[0,T]}^k(\circ\widehat{B})$ to be the $k^\text{th}$ level of the path-signature $S_{[0,T]}(\circ\widehat{B})$, i.e. the projection of the path-signature onto $(\mathbb{R}^{d+1})^{\otimes k}$. Define,
    \[
        \text{Taylor}(X_0,f,S^{(m)}_{[0,T]}(\circ\widehat{B}),\phi):=\sum_{k=0}^mf^{\circ k}\big (\phi(X_0)\big)\mathcal{S}^k_{[0,T]}(\circ\widehat{B}).
    \]
    Then $\hat{\phi}(X_T):=$Taylor$(X_0,f,S^{(m)}_{[0,T]}(\circ\widehat{B}))$ is an approximation for $\phi(X_T)$. Moreover,
    \[
        \max_{x\in \mathbb{R}^n}\sqrt{\mathbb{E}\big [\big(\phi(X_T)-\hat{\phi}(X_T)\big )^2\big ]}\leq CT^\frac{m+1}{2}\sup_{1\leq k\leq m}\big\|f^{\circ k}\big (\phi(\cdot)\big)\mathcal{S}^k_{[0,T]}(\circ\widehat{B})\big \|_\infty,
    \]
    for some constant $C$ that depends only on $d$ and $m$.
\end{theorem}
Further details and a full proof of Theorem \ref{thm:taylor} are found in Kloeden and Platen \cite{KloedenP.E.1991SaIS}. The important insight is that, given $\phi$ is suitably chosen, the error term converges to zero as $T\rightarrow 0$, at a rate that is exponential in $m$.

To utilise a cubature formula in association with Taylor's theorem, we replace the signature with the expected signature - which we have approximated using cubature.  The Taylor expansion consequently yields an approximation for $\mathbb{E}[\phi(X_T)]$, and approximating in this way will be referred to henceforth as the \textit{Taylor method}. The error term remains of order $O(T^\frac{m+1}{2})$, though this is in no way trivial and should be checked. The details for this can be found in Section 3 of Lyons and Victoir \cite{LyonsTerry2004CoWs}. Furthermore, using the Malliavin calculus in Kusuoka and Stroock \cite{KusuokaShigeo1984AotM}, it is possible to loosen the condition on $\phi$ from smooth to Lipschitz.

\subsection{Log-ODE method}
As briefly mentioned in Section \ref{sec:path-sigs}, the log-signature is often a more efficient object to work with than the path-signature itself, since it omits a number of redundant terms. This motivated the development of the Log-ODE method, a method which adapts the Taylor method. This method forms a system of ODEs by acting the Taylor expansion on the log-signature instead of the signature.
\begin{theorem}[Log-ODE method]
    Define the action of the Taylor expansion, $\hat{f}:\mathbb{R}^{d+1}\rightarrow L\big (T^N(\mathbb{R}^n),\mathbb{R}^{d+1}\big )$ by $\hat{f}(z)=\text{Taylor}(z,f,\cdot\ ,\phi)$. Define the ODE,
    \begin{align*}
        \frac{\text{d}z}{\text{d}u}&=\hat{f}(z)\log\big (S_{[0,T]}^{(m)}(\circ\widehat{B})\big );\\
        z(0)&=\phi(X_0).        
    \end{align*}
    Then, $\phi(X_T)$ is approximated by $z(1)$.
\end{theorem}
As with the Taylor method, given that $\phi$ is bounded and Lipschitz, the error is bounded above by $CT^\frac{m+1}{2}$ for some constant $C$. Proof of this result can be found in Lemma 15 of Boutaib \textit{et al.} \cite{BoutaibYouness2013DEeo}. As before, we substitute the log-signature for the expected log-signature, which we can approximate using a cubature formula. A particularly elegant consequence of this is that the lie polynomial formulation of a cubature formula already approximates the log-signature. 

The main drawback to the log-ODE method is the need to solve differential equations. If a cubature formula has size $n$, then applying the log-ODE method over a partition of size $k$ requires the solving of $n^k$ differential equations. This value can get extremely large, so running numerical experiments will require the use of efficient ODE solvers and parallelisation techniques to maintain a low runtime.

\subsection{Approximation of the path-signature}\label{sec:path-sig-approx}
Both the Taylor and Log-ODE method can approximate $\mathbb{E}[X_T-X_0]$, the first level of the expected signature. A principal observation to make is that, given a controlled differential equation of the form (\ref{eq:CDE-aug}), we can formulate a system of controlled differential equations that describe the path-signature of $S_{[0,T]}(X_t)$. The equations are defined recursively by,
\[
    \text{d}S_{[0,T]}^{(k_1,\cdots ,k_n)}(X_t)=S_{[0,T]}^{(k_1,\cdots ,k_{n-1})}(X_t)f_{k_n}(X_t^{k_n})\circ\text{d}\widehat{B}_t,
\]
where $f_k$ denotes the $k^\text{th}$ dimension of $f$. The triangular structure to this system enables the computation of higher-order signature levels using only the lower-order levels. Therefore, recursively applying either the Taylor or Log-ODE method can produce an approximation of the expected (truncated) path-signature.

Applying cubature methods iteratively in this way engenders error growth as expected signature level increases. Consequently, approximating in this way will only hold useful applications up to the first few levels and after this point the compounded error term will become too large. Nonetheless, even merely the first few levels are extremely useful, especially since Brownian motion only has a H\"older exponent of $\frac{1}{2}$.

\appendix

\section{Degree 7 expected signature of augmented Brownian motion written in the symmetrised Lyndon basis}\label{app:expansion}
The expansion of the expression,
\[
    \pi_m\Bigg (\exp\Big (\epsilon_0 +\frac{1}{2}\sum_{i=1}^d \epsilon_i\otimes\epsilon_i\Big )\Bigg ),
\]
written in the symmetrised Lyndon basis is given almost in full by Litterer \cite{LittererThesis}. Additionally required is the following lemma:
\begin{lemma}
    Expressed in the symmetrised Lyndon basis,
    \begin{align*}(\varepsilon_i^{\otimes2}&,\varepsilon_j^{\otimes2},\varepsilon_k^{\otimes2}) = \frac{2}{45}\{([\varepsilon_i,[\varepsilon_j,[[\varepsilon_j,\varepsilon_k],\varepsilon_k]]],\varepsilon_i)+([[[[\varepsilon_i,\varepsilon_k],\varepsilon_k],\varepsilon_j],\varepsilon_j],\varepsilon_i)\\
    &+(\varepsilon_j,[\varepsilon_i,[\varepsilon_i,[[\varepsilon_j,\varepsilon_k],\varepsilon_k]]])-(\varepsilon_j,[[\varepsilon_i,\varepsilon_j],[[\varepsilon_i,\varepsilon_k],\varepsilon_k]])+(\varepsilon_k,[\varepsilon_i,[\varepsilon_i,[\varepsilon_j,[\varepsilon_j,\varepsilon_k]]]])\\
    &+(\varepsilon_k,[[[\varepsilon_i,\varepsilon_j],\varepsilon_j],[\varepsilon_i,\varepsilon_k]])\}+\frac{1}{9}\{([\varepsilon_i,[\varepsilon_j,\varepsilon_k]],[\varepsilon_i,[\varepsilon_j,\varepsilon_k]])+([[\varepsilon_i,\varepsilon_k],\varepsilon_j],[\varepsilon_i,[\varepsilon_j,\varepsilon_k]])\\
    &+([[\varepsilon_i,\varepsilon_k],\varepsilon_j],[[\varepsilon_i,\varepsilon_k],\varepsilon_j])\}+\frac{2}{15}\{([[\varepsilon_i,[\varepsilon_j,\varepsilon_k]],[\varepsilon_j,\varepsilon_k]],\varepsilon_i)+(\varepsilon_j,[[\varepsilon_i,\varepsilon_k],[[\varepsilon_i,\varepsilon_k],\varepsilon_j]])\\
    &+(\varepsilon_k,[[\varepsilon_i,\varepsilon_j],[[\varepsilon_i,\varepsilon_k],\varepsilon_j]])\}+\frac{1}{9}\{([\varepsilon_i,[[\varepsilon_j,\varepsilon_k]],\varepsilon_k],[\varepsilon_i,\varepsilon_j])+([\varepsilon_i,\varepsilon_k],[\varepsilon_i,[\varepsilon_j,[\varepsilon_j,\varepsilon_k]]])\\
    &+([\varepsilon_i,\varepsilon_k],[[\varepsilon_i,[\varepsilon_j,\varepsilon_k]],\varepsilon_j])+([[[\varepsilon_i,\varepsilon_k],\varepsilon_j],\varepsilon_j],[\varepsilon_i,\varepsilon_k])+([[\varepsilon_i,\varepsilon_k],[\varepsilon_j,\varepsilon_k]],[\varepsilon_i,\varepsilon_j])\\
    &+([[[\varepsilon_i,\varepsilon_k],\varepsilon_k],\varepsilon_j],[\varepsilon_i,\varepsilon_j])+([\varepsilon_j,\varepsilon_k],[\varepsilon_i,[\varepsilon_i,[\varepsilon_j,\varepsilon_k]]])-([\varepsilon_j,\varepsilon_k],[[\varepsilon_i,\varepsilon_j],[\varepsilon_i,\varepsilon_k]])\}\\
    &+\frac{1}{15}\{([[\varepsilon_i,[[\varepsilon_j,\varepsilon_k],\varepsilon_k]],\varepsilon_j],\varepsilon_i)+([[\varepsilon_i,\varepsilon_k],[\varepsilon_j,[\varepsilon_j,\varepsilon_k]]],\varepsilon_i)-(\varepsilon_j,[[\varepsilon_i,[\varepsilon_j,\varepsilon_k]],[\varepsilon_i,\varepsilon_k]])\\
    &+(\varepsilon_k,[[\varepsilon_i,\varepsilon_j],[\varepsilon_i,[\varepsilon_j,\varepsilon_k]]])\}+\frac{4}{45}\{([[[\varepsilon_i,\varepsilon_k],[\varepsilon_j,\varepsilon_k]],\varepsilon_j],\varepsilon_i)\}+\frac{1}{18}\{([[\varepsilon_i,\varepsilon_k],\varepsilon_k],[[\varepsilon_i,\varepsilon_j],\varepsilon_j])
\end{align*}
\begin{align*}
    &+([\varepsilon_j,[\varepsilon_j,\varepsilon_k]],[\varepsilon_i,[\varepsilon_i,\varepsilon_k]])+([[\varepsilon_j,\varepsilon_k],\varepsilon_k],[\varepsilon_i,[\varepsilon_i,\varepsilon_j]])\}+\frac{1}{45}\{(\varepsilon_j,[\varepsilon_i,[[\varepsilon_i,\varepsilon_k],[\varepsilon_j,\varepsilon_k]]])\\
    &+(\varepsilon_j,[\varepsilon_i,[[[\varepsilon_i,\varepsilon_k],\varepsilon_k],\varepsilon_j]])+(\varepsilon_k,[\varepsilon_i,[[\varepsilon_i,[\varepsilon_j,\varepsilon_k]],\varepsilon_j]])+(\varepsilon_k,[\varepsilon_i,[[[\varepsilon_i,\varepsilon_k],\varepsilon_j],\varepsilon_j]])\}\\
    &+\frac{1}{2}\{(\varepsilon_j,\varepsilon_j,[\varepsilon_i,\varepsilon_k],[\varepsilon_i,\varepsilon_k])+([\varepsilon_j,\varepsilon_k],[\varepsilon_j,\varepsilon_k],\varepsilon_i,\varepsilon_i)+(\varepsilon_k,\varepsilon_k,[\varepsilon_i,\varepsilon_j],[\varepsilon_i,\varepsilon_j])\}\\
    &+\frac{1}{3}\{(\varepsilon_j,\varepsilon_j,[[\varepsilon_i,\varepsilon_k],\varepsilon_k],\varepsilon_i)+([[\varepsilon_j,\varepsilon_k],\varepsilon_k],\varepsilon_j,\varepsilon_i,\varepsilon_i)+(\varepsilon_k,\varepsilon_j,\varepsilon_j,[\varepsilon_i,[\varepsilon_i,\varepsilon_k]])\\
    &+(\varepsilon_k,[\varepsilon_j,[\varepsilon_j,\varepsilon_k]],\varepsilon_i,\varepsilon_i)+(\varepsilon_k,\varepsilon_k,[[\varepsilon_i,\varepsilon_j],\varepsilon_j],\varepsilon_i)+(\varepsilon_k,\varepsilon_k,\varepsilon_j,[\varepsilon_i,[\varepsilon_i,\varepsilon_j]])\}\\
    &+\frac{2}{3}\{([\varepsilon_j,\varepsilon_k],\varepsilon_j,[\varepsilon_i,\varepsilon_k],\varepsilon_i)+(\varepsilon_k,\varepsilon_j,[\varepsilon_i,\varepsilon_k],[\varepsilon_i,\varepsilon_j])-(\varepsilon_k,[\varepsilon_j,\varepsilon_k],[\varepsilon_i,\varepsilon_j],\varepsilon_i)\}\\
    &+\{(\varepsilon_k,\varepsilon_k,\varepsilon_j,\varepsilon_j,\varepsilon_i,\varepsilon_i)\}.
\end{align*}
\end{lemma}

\section{Code to verify the degree 7 cubature formula}\label{app:code}
The following Python code produces the degree 7 cubature formula and verifies the result. The code requires Python package \texttt{numpy} as well as the packages \texttt{free\textunderscore lie\textunderscore alg} \texttt{ebra}, written by Reizenstein, \texttt{https://github.com/crispitagorico/hallareas}, and \texttt{cubature\textunderscore construction}, written by myself and located at \texttt{https://anonymous.} \texttt{4open.science/r/thesis-code}.
\begin{verbatim}
import numpy as np
from free_lie_algebra import word2Elt, Word, exp, distance
from free_lie_algebra import lieProduct as prod
from cubature_construction import gaussian_cubature_7, truncateToLevel,
    verify_lhs
from cubature_construction import lie_sum as sum

d = 3
m = 7

z, gauss_lam = gaussian_cubature_7(d = d) # degree 7 cubature on Gaussian

e = [word2Elt(Word([i])) for i in range(d+1)] # vector space basis
lie_poly = [] # list to hold the cubature lie polynomials
lam = []  # list to hold the cubature weights

for index, zk in enumerate(z):
    sum0 = e[0]
    sum1 = sum([zk[i-1]*e[i] for i in range(1,d+1)])
    sum4a = sum([sum([zk[i-1]*prod(prod(e[i],e[j]),e[j]) 
        for j in range(1,d+1)]) for i in range(1,d+1)])
    sum4b = zk[0]*prod(prod(e[1],e[2]),e[2])
        +zk[1]*prod(prod(e[2],e[3]),e[3])
        +zk[2]*prod(prod(e[3],e[1]),e[1])
    sum4c = zk[0]*prod(prod(e[1],e[3]),e[3])
        +zk[1]*prod(prod(e[2],e[1]),e[1])
        +zk[2]*prod(prod(e[3],e[2]),e[2])
    sum5 = zk[0]*prod(prod(e[1],e[2]),e[3])
        +zk[1]*prod(prod(e[2],e[3]),e[1])
        +zk[2]*prod(prod(e[3],e[1]),e[2])  
    sum6 = sum([sum([sum([
        zk[i-1]*prod(prod(prod(prod(e[j],e[k]),e[k]),e[j]),e[i])
        +zk[k-1]*prod(prod(prod(prod(e[i],e[j]),e[j]),e[k]),e[i])
        +zk[i-1]*prod(prod(prod(prod(e[k],e[i]),e[j]),e[j]),e[k])
        +zk[k-1]*prod(prod(prod(prod(e[j],e[i]),e[i]),e[k]),e[j])
        +zk[j-1]*prod(prod(prod(prod(e[k],e[j]),e[i]),e[i]),e[k])
        +zk[j-1]*prod(prod(prod(prod(e[i],e[k]),e[k]),e[i]),e[j])
        for k in range(j+1,d+1)]) for j in range(i+1,d+1)])
        for i in range(1,d+1)])
    sum7 = sum([sum([sum([
        zk[j-1]*prod(prod(prod(prod(e[k],e[i]),e[i]),e[j]),e[k])
        +zk[i-1]*prod(prod(prod(prod(e[k],e[j]),e[j]),e[i]),e[k])
        +zk[k-1]*prod(prod(prod(prod(e[j],e[k]),e[i]),e[i]),e[j])
        +zk[k-1]*prod(prod(prod(prod(e[i],e[k]),e[j]),e[j]),e[i])
        for k in range(j+1,d+1)]) for j in range(i+1,d+1)])
        for i in range(1,d+1)]) 
    sum8 = sum([sum([sum([
        zk[i-1]*prod(prod(prod(prod(e[i],e[k]),e[k]),e[j]),e[j])
        +zk[i-1]*prod(prod(prod(prod(e[i],e[j]),e[j]),e[k]),e[k])
        +zk[i-1]*prod(prod(prod(prod(e[j],e[i]),e[k]),e[k]),e[j])
        +zk[j-1]*prod(prod(prod(prod(e[j],e[k]),e[k]),e[i]),e[i])
        +zk[k-1]*prod(prod(prod(prod(e[k],e[i]),e[i]),e[j]),e[j])
        +zk[j-1]*prod(prod(prod(prod(e[i],e[j]),e[k]),e[k]),e[i])
        +zk[j-1]*prod(prod(prod(prod(e[j],e[i]),e[i]),e[k]),e[k])
        +zk[k-1]*prod(prod(prod(prod(e[k],e[j]),e[j]),e[i]),e[i])
        for k in range(j+1,d+1)]) for j in range(i+1,d+1)])
        for i in range(1,d+1)])
    sum10 = sum([sum([
        zk[i-1]*prod(prod(prod(prod(e[i],e[j]),e[j]),e[j]),e[j])
        +zk[j-1]*prod(prod(prod(prod(e[j],e[i]),e[i]),e[i]),e[i]) 
        for j in range(i+1,d+1)]) for i in range(1,d+1)])
    sum11 = sum([sum([
        zk[j-1]*prod(prod(prod(prod(e[i],e[j]),e[j]),e[j]),e[i])
        +zk[i-1]*prod(prod(prod(prod(e[j],e[i]),e[i]),e[i]),e[j])
        for j in range(i+1,d+1)]) for i in range(1,d+1)])
    sum12 = sum([sum([
        zk[j-1]*prod(prod(prod(prod(e[j],e[i]),e[j]),e[i]),e[j])
        +zk[i-1]*prod(prod(prod(prod(e[i],e[j]),e[i]),e[j]),e[i])
        for j in range(i+1,d+1)]) for i in range(1,d+1)])
    sum13 = sum([prod(prod(e[0], e[i]), e[i]) for i in range(1, d+1)])
    sum14 = sum([prod(prod(prod(prod(e[0],e[i]),e[i]),e[i]),e[i]) 
        for i in range(1,d+1)])
    sum15 = (1/120)*sum([sum([
        sym_prod([prod(prod(prod(prod(e[0],e[i]),e[i]),e[j]),e[j])])
        +sym_prod([prod(prod(prod(prod(e[0],e[j]),e[j]),e[i]),e[i])])
        for j in range(i+1,d+1)]) for i in range(1,d+1)])
    sum16 = (1/24)*(2/15)*sum([sum([
        sym_prod([prod(prod(prod(prod(e[j],e[0]),e[i]),e[i]),e[j])])    
        +sym_prod([prod(prod(prod(prod(e[i],e[0]),e[j]),e[j]),e[i])])
        for j in range(i+1,d+1)]) for i in range(1,d+1)])
    sum17 = (1/24)*(1/15)*sum([sum([
        sym_prod([prod(prod(prod(prod(e[i],e[j]),e[j]),e[0]),e[i])])
        +sym_prod([prod(prod(prod(prod(e[j],e[i]),e[i]),e[0]),e[j])])
        for j in range(i+1,d+1)]) for i in range(1,d+1)])

  
    for gamma_1 in [-1,1]:
        for gamma_2 in [-1,1]:
            for gamma_3 in [-1,1]:
                for gamma_4 in [-1,1]:
                    sum2a = np.sqrt(1/12)*(gamma_1*gamma_2*gamma_4
                        -gamma_2*zk[0]+gamma_1*zk[1])*prod(e[1],e[2])
                    sum2b = np.sqrt(1/12)*(gamma_1*gamma_3*gamma_4
                        +gamma_3*zk[0]+gamma_1*zk[2])*prod(e[1],e[3])
                    sum2c = np.sqrt(1/12)*(gamma_2*gamma_3*gamma_4
                        +gamma_3*zk[1]+gamma_2*zk[2])*prod(e[2],e[3])
                    sum3 = np.sqrt(1/12)*(gamma_3*prod(e[0],e[3])
                        -gamma_2*prod(e[0],e[2])-gamma_1*prod(e[0],e[1]))
                    sum9a = prod(prod(prod(e[3],e[2]),e[3]),e[1])
                        +prod(prod(prod(e[1],e[2]),e[3]),e[3])
                        +prod(prod(prod(e[1],e[3]),e[3]),e[2])
                        +2*prod(prod(prod(e[1],e[2]),e[2]),e[2])
                        +2*prod(prod(prod(e[1],e[2]),e[1]),e[1])
                    sum9b = prod(prod(prod(e[1],e[3]),e[2]),e[2])
                        +prod(prod(prod(e[1],e[2]),e[2]),e[3])
                        +prod(prod(prod(e[2],e[3]),e[2]),e[1])
                        +2*prod(prod(prod(e[1],e[3]),e[3]),e[3])
                        +2*prod(prod(prod(e[1],e[3]),e[1]),e[1])
                    sum9c = prod(prod(prod(e[2],e[1]),e[1]),e[3])
                        +prod(prod(prod(e[1],e[3]),e[1]),e[2])
                        +prod(prod(prod(e[2],e[3]),e[1]),e[1])
                        +2*prod(prod(prod(e[2],e[3]),e[3]),e[3])
                        +2*prod(prod(prod(e[2],e[3]),e[2]),e[2])
                    lie_poly.append(sum0+sum1+sum2a+sum2b+sum2c+sum3
                        +(1/12)*sum13+(1/360)*sum14+(1/12)*sum4a
                        +(1/6)*(gamma_1*gamma_2*sum4b+gamma_1*gamma_3*sum4c)
                        +(1/6)*gamma_2*gamma_3*sum5+(1/360)*sum6
                        +(1/180)*sum7+(1/120)*sum8
                        +(1/(24*np.sqrt(3)))*(gamma_1*gamma_2*gamma_4*sum9a
                        +gamma_1*gamma_3*gamma_4*sum9b
                        +gamma_2*gamma_3*gamma_4*sum9c)+(1/360)*sum10
                        +(1/120)*sum11+(1/90)*sum12+sum15+sum16+sum17)
                    lam.append((1/2)**4*gauss_lam[index])

rhs = sum([lam[i]*exp(lie_poly[i], maxLevel = m) for i in range(len(lie_poly))]) 
rhs = truncateToLevel(rhs, m)
print(distance(rhs, verify_lhs(m,d)))
\end{verbatim}

\bibliography{refs}
\bibliographystyle{alpha}

\end{document}